\documentclass[oneside,a4paper,english,10pt]{smfcustom}

\usepackage{mathrsfs,mathtools}
\usepackage[english]{babel}
\usepackage[utf8]{inputenc}
\usepackage[T1]{fontenc}
\usepackage{pxfonts}
\usepackage{xcolor}\usepackage[all]{xy}
  \def\O{\mathcal{O}} \def\Z{\mathbb{Z}}
\def\P{\mathbf{P}}
\def\Q{\mathbf{Q}}
\def\F{\mathbf{F}}
\def\tprod{\mathop{\textstyle\prod}}
\def\tsum{\mathop{\textstyle\sum}}
\def\bydef{\coloneqq}
\DeclareMathOperator{\rk}{rank}
\DeclareMathOperator{\Hom}{Hom}

\NoSwapTheoremNumbers
\author{Lionel Darondeau}
\address{Institute of Mathematics, Polish Academy of Sciences.}
\email{lionel.darondeau@normalesup.org}
\author[Piotr Pragacz]{Piotr Pragacz\(^{\ast}\)}
\address{Institute of Mathematics, Polish Academy of Sciences.}
\email{P.Pragacz@impan.pl}
\thanks{\(^{\ast}\) Supported by National Science Center (NCN) grant no. 2014/13/B/ST1/00133}
\title{Universal Gysin formulas for flag bundles}
\date{July 19, 2016. {v4}}
\dedicatory{To the memory of Alain Lascoux}
\begin{document}
\begin{abstract}
  We give push-forward formulas for all flag bundles of types \(A\), \(B\), \(C\), \(D\). The formulas (and also the proofs) involve only Segre classes of the original vector bundles and characteristic classes of universal bundles.
  As an application, we provide new determinantal formulas.
\end{abstract}
\maketitle
\mainmatter
\section{Introduction}
A proper morphism \(F\colon Y\to X\) of non-singular algebraic varieties over an algebraically closed field yields an additive map \(F_{\ast}\colon A^{\bullet}Y\to A^{\bullet}X\) of Chow groups induced by push-forward cycles, called the \textsl{Gysin map} (see Fulton's book~\cite{Fulton};
note that the theory developed in this book allows one to generalize the results of the present paper to singular varieties over a field and their Chow groups; moreover, for complex varieties, one can also use the cohomology rings with integral coefficients).
We will alternatively denote \(F_{\ast}\) by \(\int_{Y}^{X}\).

Push-forward formulas show how the classes of algebraic cycles on \(Y\) go via the Gysin map to classes of algebraic cycles on \(X\). 
In the present paper, we are interested in push-forwards in flag bundles. 
We shall give formulas for the classical types \(A\), \(B\), \(C\), \(D\). These have a \emph{universal} character in three aspects:
\begin{itemize}
  \item these involve characteristic classes of universal vector bundles;
  \item these universally hold for \emph{any} polynomial in such classes;
  \item these use in a universal way only the Segre classes of the original vector bundles.
\end{itemize}

The starting point of our argument is a reformulation of the classical formula for push-forward of powers of the hyperplane class in a projective bundle, that we recall.
Let \(E\to X\) be a vector bundle of rank \(n\) on a variety \(X\). 
Let \(\P(E)\to X\) be the projective bundle of lines in \(E\) 
and let \(\xi\bydef c_{1}(\O_{\P(E)}(1))\) be the hyperplane class.
For any \(i\), the \(i\)th \textsl{Segre class of \(E\)} is
\[
  s_{i}(E)
  \bydef
  \int_{\P(E)}^{X}\xi^{i+n-1}.
\]
This is a definition in~\cite{Fulton} and a lemma in preceding intersection theory (see \textit{e.g.}~\cite[Lemma 1]{JLP}).

Then, consider a polynomial \(f(\xi)=\sum_{i}\alpha_{i}\xi^{i}\) with coefficients \(\alpha_{i}\) in the Chow ring of \(X\) (here, we identify \(A^{\bullet}X\) with a subring of \(A^{\bullet}\P(E)\) and throughout the text, we will often omit pullback notation for vector bundles and algebraic cycles). Using the projection formula
\begin{equation}
  \label{eq:push1}
  \int_{\P(E)}^{X}
  f(\xi)
  =
  \sum_{i}
  \alpha_{i}\,
  s_{i-n+1}(E).
\end{equation}

This formula can be transformed using the point of view of generating series.
Before going further, let us introduce some notation.
For a monomial \(m\) and a Laurent polynomial \(f\), we will denote by \([m](f)\) the coefficient of \(m\) in \(f\) and we will call \(m\) the \textsl{extracted monomial}.
It is clear that for any \textsl{shifting monomial} \(\tilde{m}\)
\[
  [\tilde{m}m](\tilde{m}f)
  =
  [m](f).
\]
We will use this property repeatedly.

Coming back to formula~\eqref{eq:push1}, let
\[
  s_{x}(E)
  \bydef
  1
  +
  xs_{1}(E)
  +
  \dotsb
  +
  x^{\dim(X)}s_{\dim(X)}(E),
\]
be the \textsl{Segre polynomial} of \(E\).
Thus, by definition
\[
  s_{i-n+1}(E)
  =
  [x^{i-n+1}]\big(s_{x}(E)\big)
  =
  [x^{-n+1}]\big(x^{-i}s_{x}(E)\big).
\]
In order to have non-negative powers we will use the change of variables \(t=1/x\).
We get
\[
  s_{i-n+1}(E)
  =
  [t^{n-1}]\big(t^{i}s_{1/t}(E)\big).
\]
Remark that in the later expression, the extracted monomial does \emph{not} depend on \(i\), whence, by linearity, 
for any polynomial \(f\in A^{\bullet}X[t]\),
the push-forward formula~\eqref{eq:push1} becomes
\begin{equation}
  \label{eq:PE}
  \int_{\P(E)}^{X}
  f(\xi)
  =
  [t^{n-1}]
  \big(
  \tsum_{i}
  \alpha_{i}t^{i}\,
  s_{1/t}(E))
  \big)
  =
  [t^{n-1}]
  \big(
  f(t)
  s_{1/t}(E))
  \big).
\end{equation}
Notably, this expression~\eqref{eq:PE} involves only the Segre polynomial, that behave better than individual Segre classes with respect to relations in the Grothendieck ring of the base variety.
This will play a significant role in our induction strategy
to generalize the fundamental formula~\eqref{eq:PE} to all partial flag bundles of types \(A,B,C,D\). 

This is done in Theorems \ref{thm:A}, \ref{thm:C}, \ref{thm:BD}.
The looked at push-forwards are presented as suitable coefficients of some polynomials in the ring \(A^{\bullet}X[t_{1},\dotsc,t_{d}]\), where \(\{t_{i}\}\) are some auxiliary variables and \(d\) is an integer determined by the flag data.
The proof relies on the idea to iterate formula~\eqref{eq:PE} on chains of projective bundles
(for types \(B\) and \(D\), we assume that there exist isotropic subbundles of maximal possible dimension, \textit{cf.}~\cite{EG,F3} for a discussion on that assumption).
This allows us to give a formula for full flag bundles. The case of a general flag bundle is then obtained similarly as in Damon's paper \cite{Damon}, without using Grassmann bundles.

To give an insight, we give three first examples.
All unexplained notation can be found in Sections \ref{se:A}, \ref{se:C}, \ref{se:BD}.
Let us however quickly mention that throughout this text, when we consider a certain flag bundle \(F\to X\), 
the \(\xi_{i}\) are the Chern roots of (the dual of) the tautological subbundles on \(F\);
the letter \(f\) denotes a polynomial in the indicated number of variables with coefficients in \(A^{\bullet}X\) (recall that we identify \(A^{\bullet}X\) with subrings of the Chow rings of flag bundles);
and by tacit assumption \(f(\xi_{1},\dotsc,\xi_{d})\in A^{\bullet}F\), \textit{i.e.} the polynomial \(f\) has the appropriate symmetries.
These examples are:
\begin{itemize}
  \item for the Grassmann bundle \(\F(d)(E)\to X\) of a rank \(n\) vector bundle,
    \[
      \int_{\F(d)(E)}^{X}
      f(\xi_{1},\dotsc,\xi_{d})
      =
      \big[\tprod_{i=1}^{d}
        t_{i}^{(n-i)}
      \big]
      \bigg(
      f(t_{1},\dotsc,t_{d})
      \tprod_{1\leq i< j\leq d}
      (t_{i}-t_{j})
      \tprod_{1\leq i\leq d}s_{1/t_{i}}(E)
      \bigg);
    \]
  \item for the complete flag bundle \(\F(E)\to X\) of a rank \(n\) vector bundle,
    \[
      \int_{\F(E)}^{X}
      f(\xi_{1},\dotsc,\xi_{n-1})
      =
      \big[\tprod_{i=1}^{n-1}t_{i}^{(n-1)}\big]
      \bigg(
      f(t_{1},\dotsc,t_{n-1})
      \tprod_{1\leq i< j\leq n-1}
      (t_{i}-t_{j})
      \tprod_{1\leq i\leq n-1}
      s_{1/t_{i}}(E)
      \bigg);
    \]
  \item for the symplectic Grassmann bundle \(\F^{\omega}(d)(E)\to X\) of a rank \(2n\) symplectic vector bundle 
    (assuming here that the symplectic form \(\omega\) has values in \(\O_{X}\)),
    \[
      \int_{\F^{\omega}(d)(E)}^{X}
      f(\xi_{1},\dotsc,\xi_{d})
      =
      \big[\tprod_{i=1}^{d}
        t_{i}^{(2n-i)}
      \big]
      \bigg(
      f(t_{1},\dotsc,t_{d})
      \tprod_{1\leq i< j\leq d}
      (t_{i}^{2}-t_{j}^{2})
      \tprod_{1\leq i \leq d}
      s_{1/t_{i}}(E)
      \bigg).
    \]
\end{itemize}

The reformulation~\eqref{eq:PE} and the idea to iterate the obtained formula on chains (or ``towers'') of projective bundles originally appeared in the paper~\cite{D} by the first author.
The idea of generalizing this formula to flag bundles was signaled by Bérczi, and it became clear that this suggestion was relevant as we recovered a formula for type \(A\) of Ilori from~\cite{Ilori}.
After the first version of the paper was completed, Manivel informed us that, independently, in their recent paper~\cite{KT}, Kaji and Terasoma prove a formula for type \(A\) (in the particular case of full flag bundles). 

There exist various approaches to push-forward formulas for flag bundles:
\begin{itemize}
  \item using Grothendieck residues (Akyildiz--Carrell~\cite{AC}, Damon~\cite{Damon}, Quillen~\cite{Quillen});
  \item using localization and residues at infinity (Bérczi--Szenes~\cite{BS}, Tu~\cite{Tu}, Zielenkiewicz~\cite{Zielenkiewicz,Zie2});
  \item using symmetrizing operators (Brion~\cite{Brion}, the second author, \textit{e.g.}~\cite{P3,PP} and Ratajski~\cite{PR});
  \item using Schur functions and Grassmann extensions (Józefiak, Lascoux and the second author~\cite{JLP}; for supersymmetric functions see~\cite{P3, FP});
  \item using residues and Grassmann extensions (Kazarian~\cite{Kaz2,Kazarian}), which leads to formulas showing similarities with ours.
\end{itemize}
Also, in the recent paper~\cite{P2} by the second author, a deformation of a push-forward formula from~\cite{JLP} is shown for Hall--Littlewood polynomials.

In the present paper we use only elementary linear algebra of polynomials, whose coefficients yield the sought push-forwards.
A use of Segre polynomials \(s_{x}(E)\), specialized with \(x=1/t\), leads to remarkable compact expressions. Moreover, it allows us to generalize the formula for type \(A\) uniformly to other classical types \(B\), \(C\), \(D\), for which, to the best of our knowledge, no general Gysin formula was known.

To give a first illustration of the usefulness of our formulas, in Section~\ref{se:formulas}, we provide new determinantal formulas for the push-forward of monomials and of Schur classes.
In a forthcoming paper, we will also apply our formulas in order to compute the fundamental classes of Schubert varieties for the four classical types.

\section{Universal push-forward formulas for ordinary flag bundles}
\label{se:A}
We first consider type \(A\).
\subsection{Definition of partial flag bundles of type \ensuremath{A}}
\label{sse:defA}
Let \(E\to X\) be a rank \(n\) vector bundle.  
Let \(1\leq d_{1}<\dotsb< d_{m}\leq n-1\) be a sequence of integers. 
We denote by \(\pi\colon\F(d_{1},\dotsc,d_{m})(E)\to X\) 
the bundle of flags of subspaces of dimensions \(d_{1},\dotsc,d_{m}\) in the fibres of \(E\).  
On \(\F(d_{1},\dotsc,d_{m})(E)\), there is a universal flag 
\(U_{d_{1}}\subsetneq\dotsb\subsetneq U_{d_{m}}\)
of subbundles of \(\pi^{\ast}E\), where \(\rk(U_{d_{k}})=d_{k}\)
(the fiber of \(U_{d_{k}}\) over the point 
\((V_{d_{1}}\subsetneq\dotsb\subsetneq V_{d_{m}}\subset E_{x})\), 
where \(x\in X\), is equal to \(V_{d_{k}}\)). 
For a foundational account on flag bundles, see~\cite{Groth2}.

\subsection{Step-by-step construction of full flag bundles}
\label{sse:flagA}
In order to use formula~\eqref{eq:PE}, we recall the construction of 
the flag bundles \(\F(1,2,\dotsc,d)(E)\)  for \(d=1,2,\dotsc,n-1\) as the chains of
projective bundles of lines~\cite{Groth1} (see also~\cite[\S2.2]{FP}). We proceed over \(x\in X\) 
and write \(E=E_{x}\). Consider a flag of subspaces
\[
  0=V_{0}
  \subsetneq
  V_{1}
  \subsetneq
  \dotsb
  \subsetneq
  V_{n-1}
  \subsetneq
  V_{n}=E,
\]
such that for each \(i\), the dimension of \(V_{i}\) is \(i\).
For \(V_{1}\), we can take any line in \(E\). It follow that \(\F(1)(E)\simeq \P(E)\). 
Next, we consider \(\F(1,2)(E)\to\F(1)(E)\) above \(V_{1}\). 
In order to get a \(2\)-dimensional subspace \(V_{1}\subsetneq V_{2}\subsetneq E\), it suffices to pick one more line in \(\P(E/V_{1})\). 
We iterate this construction, 
\[
  \begin{array}{c}
    \F(1,\dotsc,i,i+1)(E)\\
    \downarrow\\
    \F(1,\dotsc,i)(E)
  \end{array}
  (\text{fiber}=\P(E/V_{i}))
\]
picking one line in \(\P(E/V_{i})\) at each step, until \(E/V_{n-1}\) is \(1\)-dimensional, so \(\P(E/V_{n-1})\) is a point.

Globalizing this construction over \(X\), we obtain
a chain of projective bundles of lines
\[
  \F(E)
  \bydef
  \F(1,\dotsc,n-1)(E)
  \to
  \F(1,\dotsc,n-2)(E)
  \to
  \dotsb
  \to
  \F(1,2)(E)
  \to
  \F(1)(E)
  \to
  X,
\]
which is the same as
\[
  \F(E)
  \bydef
  \P(E/U_{n-1})
  \to
  \P(E/U_{n-2})
  \to
  \dotsb
  \to
  \P(E/U_{1})
  \to
  \P(E)
  \to
  X.
\]

In this paper, the flag bundles \(\F(1,\dotsc,d)(E)\), for \(d=1,\dotsc,n-1\), are termed \textsl{full}, since these involve all first consecutive integers up to a certain \(d\) and we call \textsl{complete} flag bundle, denoted \(\F(E)\), the full flag bundle  if \(d=n-1\). Note that this terminology may vary in the literature.

\subsection{Universal push-forward formula for flag bundles}
Let \(E\to X\) be a vector bundle of rank \(n\).
Given a sequence of integers \(0=d_{0}<d_{1}<\dotsb<d_{m}\leq n-1\) as in Sect.~\ref{sse:defA}, 
we set \(d\bydef d_{m}\) and write for  \(k=1,\dotsc,m\)
\[
  r_{k}
  \bydef
  d_{k}-d_{k-1}
  =
  \rk(U_{d_{k}}/U_{d_{k-1}}).
\]
For \(i=1,\dotsc,d\), we denote \(\xi_{i}\) the hyperplane class on \(\P(E/U_{d-i})\)
\[
  \xi_{i}
  \bydef
  -c_{1}(U_{d+1-i}/U_{d-i}).
\]
The classes \(\xi_{1},\dotsc,\xi_{d}\) are Chern roots of \(U_{d}^{\vee}\)
and generate the cohomology of the fibers of 
\begin{equation}
  \label{eq:roots}
  \F(1,\dotsc,d)(E)
  \stackrel{\substack{\xi_{1}\\\downarrow}}\longrightarrow
  \F(1,\dotsc,d-1)(E)
  \stackrel{\substack{\xi_{2}\\\downarrow}}\longrightarrow
  \dotsb\dotsb
  \stackrel{\substack{\xi_{d-1}\\\downarrow}}\longrightarrow
  \F(1)(E)
  \stackrel{\substack{\xi_{d}\\\downarrow}}\longrightarrow
  X
\end{equation}
in the above chain of projective bundles.

The following Gysin formula holds for the partial flag bundle \(\F(d_{1},\dotsc,d_{m})(E)\to X\).
\begin{theo}
  \label{thm:A}
  For any rational equivalence class
  \(
  f(\xi_{1},\dotsc,\xi_{d})
  \in
  A^{\bullet}(\F(d_{1},\dotsc,d_{m})(E)),
  \)
  one has
  \[
    \int_{\F(d_{1},\dotsc,d_{m})(E)}^{X}
    f(\xi_{1},\dotsc,\xi_{d})
    =
    \Big[
      {t_{1}}^{e_{1}}\dotsm {t_{d}}^{e_{d}}
    \Big]
    \bigg(
    f(t_{1},\dotsc,t_{d})\,
    \tprod_{1\leq i<j\leq d}
    (t_{i}-t_{j})
    \tprod_{1\leq i \leq d}
    s_{1/t_{i}}(E)
    \bigg),
  \]
  where for 
  \(j=d-d_{k}+i\) with \(i=1,\dotsc,r_{k}\),
  we denote
  \(e_{j}\bydef n-i\).
\end{theo}

\begin{proof}
  The idea of the proof is to iterate formula~\eqref{eq:PE} on the above chain of projective bundles of lines.

  We first consider full flags. The case of \(\F(1)(E)=\P(E)\to X\) was already treated in the introduction.
  We now generalize formula~\eqref{eq:PE} to the case of full flag bundles \(\F(1,\dotsc,d)(E)\to X\), by induction on \(d\geq1\).

  By construction, the flag bundle \(\F(1,\dotsc,d)(E)\) is the total space of the projective bundle of lines of the tautological quotient bundle \(E/U_{d-1}\to\F(1,\dotsc,d-1)(E)\).
  This quotient bundle having rank 
  \((n-d+1)\),
  a plain application of formula~\eqref{eq:PE} yields
  \[
    \int_{\F(1,\dotsc,d)(E)}^{X}
    f(\xi_{1},\dotsc,\xi_{d})
    =
    [{t_{1}}^{n-d}]
    \left(
    \int_{\F(1,\dotsc,d-1)(E)}^{X}
    f(t_{1},\xi_{2},\dotsc,\xi_{d})\,
    s_{1/t_{1}}(E/U_{d-1})
    \right).
  \]

  In order to use induction, we need to express the total Segre class \(s(E/U_{d-1})\) in terms of the remaining classes \(\xi_{2},\dotsc,\xi_{d}\),
  which is easily done by use of the Whitney sum formula for the following relation in the Grothendieck group
  of \(\F(1,\dotsc,d-1)(E)\)
  \[
    [E/U_{d-1}]
    =
    [E]
    -
    \big([U_{d-1}/U_{d-2}]+\dotsb+[U_{1}/U_{0}]\big)
  \]
  which yields
  \[
    s(E/U_{d-1})
    =
    \prod_{j=2}^{d}
    (1-\xi_{j})\,
    s(E).
  \]
  The reformulation of this formula in terms of generating series is
  \[
    s_{1/t_{1}}(E/U_{d-1})
    =
    \prod_{j=2}^{d}
    \frac{(t_{1}-\xi_{j})}{t_{1}}\,
    s_{1/t_{1}}(E).
  \]
  Hence the above formula becomes
  \[
    \int_{\F(1,\dotsc,d)(E)}^{X}
    \hskip-5pt
    f(\xi_{1},\xi_{2},\dotsc,\xi_{d})
    =
    [{t_{1}}^{n-d}]
    \left(
    \int_{\F(1,\dotsc,d-1)(E)}^{X}
    \hskip-5pt
    f(t_{1},\xi_{2},\dotsc,\xi_{d})\,
    (1/{t_{1})^{d-1}}
    \tprod_{1<j\leq d}
    (t_{1}-\xi_{j})\,
    s_{1/t_{1}}(E)
    \right),
  \]
  which it is good to reshape by multiplication of both the extracted monomial and the series by the shifting monomial \({t_{1}}^{d-1}\)
  \[
    \int_{\F(1,\dotsc,d)(E)}^{X}
    \hskip-5pt
    f(\xi_{1},\xi_{2},\dotsc,\xi_{d})
    =
    [{t_{1}}^{n-1}]
    \left(
    \int_{\F(1,\dotsc,d-1)(E)}^{X}
    \hskip-5pt
    f(t_{1},\xi_{2},\dotsc,\xi_{d})
    \tprod_{1<j\leq d}
    (t_{1}-\xi_{j})\,
    s_{1/t_{1}}(E)
    \right).
  \]

  After the first step, we have to integrate the polynomial in the variables 
  \(\xi_{2},\dotsc,\xi_{d}\)
  \[
    f(t_{1},\xi_{2},\dotsc,\xi_{d})
    \tprod_{1<j\leq d}
    (t_{1}-\xi_{j})\,
    s_{1/t_{1}}(E)
  \]
  along the fiber of the projective bundle of lines \(\F(1,\dotsc,d-1)(E)\to\F(1,\dotsc,d-2)(E)\). 

  Iterating the same reasoning until we have replaced all classes \(\xi_{i}\) by some formal variables \(t_{i}\), we obtain the announced expression
  \[
    \int_{\F(1,\dotsc,d)(E)}^{X}
    f(\xi_{1},\dotsc,\xi_{d})
    =
    [t_{1}^{n-1}\dotsm t_{d}^{n-1}]
    \bigg(
    f(t_{1},\dotsc,t_{d})
    \tprod_{1\leq i<j\leq d}
    (t_{i}-t_{j})
    \tprod_{1\leq i \leq d}
    s_{1/t_{i}}(E)
    \bigg).
  \]

  \paragraph{A useful particular case}
  In the particular case of the complete flag bundle \(\F(E)\to X\) of a rank \(r\) vector bundle \(E\to X\) on a variety \(X\), 
  and for the polynomial
  \[
    g_{r}(t_{1},\dotsc,t_{r-1})
    \bydef
    {t_{1}}^{1}{t_{2}}^{2}\dotsm{t_{r-1}}^{r-1},
  \]
  the formula reads
  \begin{equation}
    \label{lem:aux_complete_flag}
    \int_{\F(E)}^{X}
    g_{r}(\xi_{1},\dotsc,\xi_{r-1})
    =
    [X].
  \end{equation}

  Indeed, in the formula of Theorem~\ref{thm:A}, 
  \[
    \int_{\F(E)}^{X}
    g_{r}(\xi_{1},\dotsc,\xi_{r-1})
    =
    [t_{1}^{r-1}\dotsm t_{r-1}^{r-1}]
    \bigg(
    \tprod_{1\leq i\leq r-1}
    {t_{i}}^{i}
    \tprod_{1\leq i<j\leq r-1}
    (t_{i}-t_{j})
    \tprod_{1\leq i \leq r-1}
    s_{1/t_{i}}(E)
    \bigg),
  \]
  we have chosen the degree of \(g_{r}\) so that the homogeneous degree of the product of the two first factors
  is exactly the homogeneous degree of the extracted monomial. 
  As a consequence, only the constant term of the third factor \(\prod s_{1/t_{i}}(E)\) can contribute. This coefficient is \(1=[X]\in A^{0}X\), whence
  \[
    \int_{\F(E)}^{X}
    g_{r}(\xi_{1},\dotsc,\xi_{r-1})
    =
    [t_{1}^{r-1}\dotsm t_{r-1}^{r-1}]
    \bigg(
    \tprod_{1\leq i\leq r-1}
    {t_{i}}^{i}
    \tprod_{1\leq i<j\leq r-1}
    (t_{i}-t_{j})
    \bigg)
    [X].
  \]
  For \(r=2\), the coefficient of \([X]\) is obviously \(1\). For \(r>2\), after shifting both the extracted monomial and the series by the monomial \({t_{r-1}}^{r-1}\), one is led to extract the part of homogeneous degree \(0\) with respect to \(t_{r-1}\).
  It amounts to replace \(\prod_{i=1}^{r-2}(t_{i}-t_{r-1})\) by 
  \(\prod_{i=1}^{r-2}t_{i}\).
  Then, shifting by this monomial 
  one gets
  \[
    [t_{1}^{r-1}\dotsm t_{r-1}^{r-1}]
    \bigg(
    \tprod_{1\leq i\leq r-1}{t_{i}}^{i}
    \tprod_{1\leq i<j\leq r-1}
    (t_{i}-t_{j})
    \bigg)
    =
    [t_{1}^{r-2}\dotsm t_{r-2}^{r-2}]
    \bigg(
    \tprod_{1\leq i\leq r-2}
    {t_{i}}^{i}
    \tprod_{1\leq i<j\leq r-2}
    (t_{i}-t_{j})
    \bigg).
  \]
  This is the adequate induction formula.

  \paragraph{From full flags to partial flags}
  We follow here~\cite{Damon}.
  Let \(0=d_{0}<d_{1}<\dotsb<d_{m}=d\) be an increasing sequence of integers.
  Recall that on \(F\bydef\F(d_{1},\dots,d_{m})(E)\), there is the universal flag of vector bundles
  \[
    0\subsetneq U_{d_{1}}\subsetneq\dotsb\subsetneq U_{d_{m}}\subsetneq E,
  \]
  where \(\rk(U_{d_{i}})=d_{i}\).
  The fiber product
  \[
    \mathbf{Y}
    \bydef
    \F(U_{d_{1}})\times_{F}\F(U_{d_{2}}/U_{d_{1}})\times_{F}\dotsb\times_{F}\F(U_{d_{m}}/U_{d_{m-1}})
  \]
  is isomorphic to \(\F(1,\dotsc,d)(E)\) with the natural projection map \(\F(1,\dotsc,d)(E)\to F\), and we get a commutative diagram
  \begin{equation}
    \label{eq:cd}
    \vcenter{
      \xymatrix@C=3cm{
        \F(1,\dotsc,d)(E)\ar[d]_{\pi''}\ar[r]_{\simeq}^{\theta}&
        \mathbf{Y}
        \ar[d]^{\pi'}
        \\
        X&
        F
        \ar[l]_{\pi}
      }
    }.
  \end{equation}
  The fiber of \(\pi'\) over the point \((V_{d_{1}}\subsetneq V_{d_{2}}\subsetneq\dotsb\subsetneq V_{d_{m}}\subsetneq E_{x})\in F\) is the product of complete flag varieties
  \[
    \F(V_{d_{1}})\times\F(V_{d_{2}}/V_{d_{1}})\times\dotsb\times\F(V_{d_{m}}/V_{d_{m-1}}).
  \]

  For \(k=1,\dotsc,m\) and for \(i=1,\dotsc,r_{k}-1\), let \(\eta_{(d+1-d_{k})+i}\)
  denote 
  the pullback to \(A^{\bullet}(\F(U_{d_{k}}/U_{d_{k-1}}))\)
  of the hyperplane classes of the projective bundles
  \[
    \F(1,\dotsc,r_{k}-i)(U_{d_{k}}/U_{d_{k-1}})
    \to
    \F(1,\dotsc,r_{k}-i-1)(U_{d_{k}}/U_{d_{k-1}}),
  \]
  \textit{cf.}~\eqref{eq:roots} with \(E=U_{d_{k}}/U_{d_{k-1}}\) and \(d=r_{k}-1\).

  Applying the splitting principle (\textit{cf.}~\cite{Groth1}) to each graded piece \((U_{d_{k}}/U_{d_{k-1}})_{k}\) of the ``universal'' filtration , we infer that the pullbacks 
  \(\theta^{\ast}\eta_{d-d_{k}+2},\dotsc,\theta^{\ast}\eta_{d-d_{k-1}}\) of these classes are the corresponding classes 
  \(\xi_{d-d_{k}+2},\dotsc,\xi_{d-d_{k-1}}\) in \(A^{\bullet}(\F(1,\dotsc,d)(E))\).
  Note that we do not define classes \(\eta_{d-d_{k}+1}\), but it will not play any role in the sequel of the proof.

  Our goal is to compute the push-forward \(\pi_{\ast}=\int_{F}^{X}\). Since \(\pi_{\ast}'\) is surjective, it is enough to understand \(\pi_{\ast}\circ\pi_{\ast}'\), which is equal to \(\pi_{\ast}''\circ(\theta^{-1})_{\ast}\). But from the full flag case, we know \(\pi_{\ast}'\) and also \(\pi_{\ast}''\), so we can proceed.

  We are now in position to prove the generalization of~\eqref{eq:PE} to the case of partial flag bundles.
  Let 
  \[
    g(t_{1},\dotsc,t_{d})
    \bydef
    \prod_{1\leq k\leq m}
    \prod_{1\leq i\leq r_{k}}
    {t_{(d-d_{k})+i}}^{i-1}
    =
    \prod_{1\leq k\leq m}
    g_{r_{k}}(t_{d-d_{k}+2},\dotsc,t_{d-d_{k-1}}).
  \]
  It follows from~\eqref{lem:aux_complete_flag} that
  \(
  \pi_{\ast}'g(\eta_{1},\dotsc,\eta_{d})
  =
  [X]
  \), whence, by the projection formula for \(\pi'\), we have
  \[
    \pi_{\ast}'
    \left(
    f(\xi_{1},\dotsc,\xi_{d})
    g(\eta_{1},\dotsc,\eta_{d})
    \right)
    =
    f(\xi_{1},\dotsc,\xi_{d}).
  \]
  This implies
  \[
    \pi_{\ast}
    f(\xi_{1},\dotsc,\xi_{d})
    =
    \pi_{\ast}
    \pi_{\ast}'
    \left(
    f(\xi_{1},\dotsc,\xi_{d})
    g(\eta_{1},\dotsc,\eta_{d})
    \right).
  \]
  Now, from the commutativity of~\eqref{eq:cd}, we get
  \begin{align*}
    \pi_{\ast}
    \pi_{\ast}'
    \left(
    f(\xi_{1},\dotsc,\xi_{d})
    g(\eta_{1},\dotsc,\eta_{d})
    \right)
    &=
    \pi_{\ast}''
    (\theta^{-1})_{\ast}
    \left(
    f(\xi_{1},\dotsc,\xi_{d})
    g(\eta_{1},\dotsc,\eta_{d})
    \right)
    \\
    &=
    \pi_{\ast}''
    \left(
    f(\xi_{1},\dotsc,\xi_{d})
    g(\xi_{1},\dotsc,\xi_{d})
    \right).
  \end{align*}
  This last expression, using the formula for full flag bundles, is equal to
  \[
    [t_{1}^{n-1}\dotsm t_{d}^{n-1}]
    \bigg(
    f(t_{1},\dotsc,t_{d})\,
    g(t_{1},\dotsc,t_{d})\,
    \tprod_{1\leq i \leq d}
    s_{1/t_{i}}(E)
    \tprod_{1\leq i<j\leq d}
    (t_{i}-t_{j})
    \bigg).
  \]

  In order to get the announced formula, it suffices to shift both the series and the extracted monomial by the monomial \(g(t_{1},\dotsc,t_{d})\). 
  Since
  \[
    \frac{t_{1}^{n-1}\dotsm t_{d}^{n-1}}
    {g(t_{1},\dotsc,t_{d})}
    =
    \prod_{1\leq k\leq m}
    \prod_{1\leq i\leq r_{k}}
    {t_{(d-d_{k})+i}}^{n-i}
    =
    \prod_{j=1}^{d}
    {t_{j}}^{e_{j}},
  \]
  this concludes the proof.
\end{proof}

As noticed by Ilori~\cite[pp. 629--630]{Ilori}, it can be useful to have a formula with all the Chern roots of \(E\). In this context, we denote \(\xi_{i}\) the hyperplane class of \(\P(E/U_{n-i})\), independently of \(d\).
One has
\[
  \xi_{i}
  =
  -c_{1}(U_{n-i+1}/U_{n-i}),
\]
and the Chern roots of \(U_{d}^{\vee}\) are now denoted \(\xi_{n-d+1},\dotsc,\xi_{n}\).

Given a sequence of integers \(0=d_{0}<d_{1}<\dotsb<d_{m}\leq n-1\) as in Sect.~\ref{sse:defA}, 
beside the above notation, 
we set \(d_{m+1}\bydef n\) and accordingly \(r_{m+1}\bydef n-d_{m}\).
\begin{prop}
  \label{thm:A'}
  With the above notation, one has
  \[
    \int_{\F(d_{1},\dotsc,d_{m})(E)}^{X}
    f(\xi_{1},\dotsc,\xi_{n})
    =
    \Big[{t_{1}}^{e_{1}}\dotsm{t_{n}}^{e_{n}}\Big]
    \bigg(
    f(t_{1},\dotsc,t_{n})
    \tprod_{1\leq i<j\leq n}
    (t_{i}-t_{j})
    \tprod_{1\leq i \leq n}
    s_{1/t_{i}}(E)
    \bigg),
  \]
  where for \(j=n-d_{k}+i\) with \(i=1,\dotsc,r_{k}\), we denote 
  \(e_{j}\bydef n-i\).
\end{prop}
\begin{proof}
  Firstly, we treat the complete flag bundle \(\F(E)\to X\). In the Grothendieck group of \(\F(E)\),
  \[
    [E/U_{n-1}]
    =
    [E]
    -
    \sum_{i=1}^{n-1}[U_{i}/U_{i-1}],
  \]
  which yields
  \[
    s(E/U_{n-1})
    =
    \sum_{j\geq 0}
    \xi_{1}^{j}
    =
    \prod_{j=2}^{n}
    (1-\xi_{j})\,
    s(E).
  \]
  As a consequence (we skip the detail)
  \[
    \xi_{1}^{j}
    =
    \Big[t_{1}^{n-1}\Big]
    \bigg(
    t_{1}^{j}
    \tprod_{1<j\leq n}
    (t_{1}-\xi_{j})\,
    s_{1/t_{1}}(E)
    \bigg).
  \]
  By linearity, one gets a polynomial expression in the \(n-1\) remaining classes \(\xi_{2},\dotsc,\xi_{n}\)---that are the Chern roots of the universal subbundle on \(\F(E)\)---
  \[
    f(\xi_{1},\xi_{2}\dotsc,\xi_{n})
    =
    [t_{1}^{n-1}]
    \Big(
    f(t_{1},\xi_{2},\dotsc,\xi_{n})
    \prod_{1<j\leq n}
    (t_{1}-\xi_{j})\,
    s_{1/t_{1}}(E)
    \Big),
  \]
  and an application of the formula proved in Theorem~\ref{thm:A} yields a similar formula with one more formal variable
  \[
    \int_{\F(E)}^{X}
    f(\xi_{1},\dotsc,\xi_{n})
    =
    [t_{1}^{n-1}\dotsm t_{n}^{n-1}]
    \bigg(
    f(t_{1},\dotsc,t_{n})\,
    \tprod_{1\leq i \leq n}
    s_{1/t_{i}}(E)
    \tprod_{1\leq i<j\leq n}
    (t_{j}-t_{i})
    \bigg).
  \]

  To get the formula for partial flag bundles, we adopt the same strategy as above, but consider the projection \(\F(E)\to\F(d_{1},\dotsc,d_{m})(E)\). The fiber over a point \((V_{d_{1}}\subsetneq V_{d_{2}}\subsetneq\dotsb\subsetneq V_{d_{m}}\subseteq E_{x})\) is the product of complete flag varieties
  \[
    \F(V_{d_{1}})\times\F(V_{d_{2}}/V_{d_{1}})\times\dotsb\times\F(E_{x}/V_{d_{m}}).
  \]
  One infers the stated formula, with shifted exponents.
\end{proof}

\section{Universal push-forward formulas for symplectic flag bundles}
\label{se:C}
Let now us deal with the symplectic setting. It can be regarded as the most simple case after the \(A\) case, as any line is isotropic for the symplectic form. This last fact will no longer be true in orthogonal setting.

\subsection{Definition of partial isotropic flag bundles of type \ensuremath{C}} 
\label{sse:defC}
Let \(E\to X\) be a rank \(2n\) vector bundle equipped with a non-degenerate symplectic form \(\omega\colon E\otimes E\to L\) (with values in a certain line bundle \(L\to X\)). We say that a subbundle \(S\) of \(E\) is isotropic if \(S\) is a subbundle of its symplectic complement \(S^{\omega}\), where
\[
  S^{\omega}
  \bydef
  \{w\in E\mid \forall v\in S\colon \omega(w,v)=0\}.
\]

Let \(1\leq d_{1}<\dotsb<d_{m}\leq n\) be a sequence of integers. 
We denote by \(\pi\colon\F^{\omega}(d_{1},\dotsc,d_{m})(E)\to X\) the bundle of flags of isotropic subspaces of dimensions \(d_{1},\dotsc,d_{m}\) in the fibers of \(E\). 
On \(\F^{\omega}(d_{1},\dotsc,d_{m})(E)\), there is a universal flag \(U_{d_{1}}\subsetneq\dotsb\subsetneq U_{d_{m}}\) of subbundles of \(\pi^{\ast}E\), where \(\rk(U_{d_{k}})=d_{k}\).  

\subsection{Step-by-step construction of isotropic full flag bundles}
\label{sse:flagC}
Let us recall the construction of \(\F^{\omega}(1,2,\dotsc,d)(E)\) for \(d=1,2,\dotsc,n\) as chains of projective bundles of lines~\cite[\S6.1]{FP}.
We proceed over \(x\in X\) and write \(E=E_{x}\). 
Consider a flag of subspaces
\[
  0=V_{0}
  \subsetneq
  V_{1}
  \subsetneq
  \dotsb
  \subsetneq
  V_{n}
  \subsetneq
  E,
\]
of \(E\) such that for each \(i\), the dimension of \(V_{i}\) is \(i\) and \(V_{i}\) is isotropic. In particular, \(V_{n}=V_{n}^{\omega}\) is a maximal isotropic subspace of \((E,\omega)\). For \(V_{1}\), we can take any line in \(E\), since \(\omega\) is skew-symmetric. It follows that \(\F^{\omega}(1)(E)\simeq \P(E)\). 
Next, we consider \(\F^{\omega}(1,2)(E)\to\F^{\omega}(1)(E)\) above \(V_{1}\). 
In order to get an isotropic subspace \(V_{1}\subsetneq V_{2}\subsetneq V_{1}^{\omega}\), it suffices to pick one more line in \(\P(V_{1}^{\omega}/V_{1})\). 
Iterating this construction, 
\[
  \begin{array}{c}
    \F^{\omega}(1,\dotsc,i,i+1)(E)\\
    \downarrow\\
    \F^{\omega}(1,\dotsc,i)(E)
  \end{array}
  (\text{fiber}=\P(V_{i}^{\omega}/V_{i})),
\]
picking one line in \(\P(V_{i}^{\omega}/V_{i})\) at each step, one ends up with 
\(V_{n}^{\omega}/V_{n}\), which is zero dimensional, thus obtaining a maximal isotropic subspace \(V_{n}\).

Globalizing this construction over \(X\), we obtain
a chain of projective bundles of lines
\[
  \F^{\omega}(1,\dotsc,n)(E)
  \to
  \F^{\omega}(1,\dotsc,n-1)(E)
  \to
  \dotsb
  \to
  \F^{\omega}(1,2)(E)
  \to
  \F^{\omega}(1)(E)
  \to
  X,
\]
which is the same as
\[
  \P(U_{n-1}^{\omega}/U_{n-1})
  \to
  \P(U_{n-2}^{\omega}/U_{n-2})
  \to
  \dotsb
  \to
  \P(U_{1}^{\omega}/U_{1})
  \to
  \P(E)
  \to
  X.
\]

\subsection{Useful relations in the Grothendieck group}
Since \(\omega\) is everywhere non-degenerate, one can consider the isomorphism
\(
\iota_{\omega}
\colon
v\in E\mapsto \omega(v,\cdot)\in\Hom(E,L)
\).
Restricting the map \(\omega(v,\cdot)\) to \(U_{1}\), one obtains an isomorphism
\[
  \iota_{\omega}
  \colon
  v\in E/U_{1}^{\omega}\mapsto \omega(v,\cdot)\in\Hom(U_{1},L),
\]
which yields the relation in the Grothendieck group of \(\F^{\omega}(1)(E)\)
\[
  [U_{1}^{\omega}]
  =
  [E]-[U_{1}^{\vee}\otimes L].
\]
It follows that
\[
  [U_{1}^{\omega}/U_{1}]
  =
  [E]
  -[U_{1}]
  -[U_{1}^{\vee}\otimes L].
\]
For \(i=2,\dotsc,n\), doing all the same reasoning on \(U_{i-1}^{\omega}\) instead of \(E\), one obtains, in the Grothendieck group of \(\F^{\omega}(1,\dotsc,i)(E)\)
\[
  [U_{i}^{\omega}/U_{i}]
  =
  [U_{i-1}^{\omega}/U_{i-1}]
  -
  [U_{i}/U_{i-1}]
  -
  [(U_{i}/U_{i-1})^{\vee}\otimes L].
\]
Note that this formula also applies to \(i=1\) with the natural convention \(U_{0}=0\).
An induction yields
\begin{equation}
  \label{eq:Grothendieck_ring}
  [U_{j}^{\omega}/U_{j}]
  =
  [E]
  -
  \sum_{i=1}^{j}
  \big(
  [U_{i}/U_{i-1}]
  +
  [(U_{i}/U_{i-1})^{\vee}\otimes L]
  \big).
\end{equation}

\subsection{Universal push-forward formula for the symplectic case}
Let \((E,\omega)\to X\) be a vector bundle having even rank \(2n\), 
equipped with a everywhere non-degenerate symplectic form \(\omega\colon E\otimes E\to L\), for a certain line bundle \(L\to X\).

Given a sequence of integers \(0=d_{0}<d_{1}<\dotsb<d_{m}\leq n\) as in Sect.~\ref{sse:defC}, 
we set \(d\bydef d_{m}\) and for \(k=1,\dotsc,m\) we write \(r_{k}\bydef d_{k}-d_{k-1}\).
We also set for \(i=1,\dotsc,d\)
\[
  \xi_{i}
  \bydef
  -c_{1}\big(U_{d+1-i}/U_{d-i}\big).
\]

The following Gysin formula holds for the isotropic partial flag bundle \(\F^{\omega}(d_{1},\dotsc,d_{m})(E)\to X\).
\begin{theo}
  \label{thm:C}
  For any rational equivalence class
  \(
  f(\xi_{1},\dotsc,\xi_{d})
  \in
  A^{\bullet}(\F^{\omega}(d_{1},\dotsc,d_{m})(E)),
  \)
  one has
  \[
    \int_{\F^{\omega}(d_{1},\dotsc,d_{m})(E)}^{X}
    f(\xi_{1},\dotsc,\xi_{d})
    =
    \Big[
      {t_{1}}^{e_{1}}\dotsm{t_{d}}^{e_{d}}
    \Big]
    \bigg(
    f(t_{1},\dotsc,t_{d})\,
    \tprod_{1\leq i<j\leq d}
    (t_{i}-t_{j})(t_{i}+t_{j}+c_{1}(L))
    \tprod_{1\leq i \leq d}
    s_{1/t_{i}}(E)
    \bigg),
  \]
  where for \(j=d-d_{k}+i\) with \(i=1,\dotsc,r_{k}\),
  we denote
  \(e_{j}\bydef2n-i\).
\end{theo}
\begin{proof}
  We prove the formula for type \(C\) in the very same way as we already did for type \(A\).

  We first consider full flags.
  We will prove the formula for \(\F^{\omega}(1,\dotsc,d)(E)\to X\) by induction on \(d=1,\dotsc,n\). Note that for \(d=1\) the sought formula is exactly the same as formula~\eqref{eq:PE}. Thus the result holds. Then for \(1<d\leq n\), we consider the projection \(\F^{\omega}(1,\dotsc,d)(E)\to \F^{\omega}(1,\dotsc,d-1)(E)\).
  Using the formula~\eqref{eq:PE} one states
  \[
    \int_{\F^{\omega}(1,\dotsc,d)(E)}^{\F^{\omega}(1,\dotsc,d-1)(E)} 
    f(\xi_{1},\dotsc,\xi_{d})
    =
    \Big[t_{1}^{2(n-d)+1}\Big]
    \bigg(
    f(t_{1},\xi_{2},\dotsc,\xi_{d})\,
    s_{1/t_{1}}(U_{d-1}^{\omega}/U_{d-1})
    \bigg).
  \]
  In order to use induction,
  it remains to express the total Segre class
  \(s(U_{d-1}^{\omega}/U_{d-1})\)
  as a polynomial in 
  \(\xi_{2},\dotsc,\xi_{d}\).
  In that aim, we use the relation~\eqref{eq:Grothendieck_ring} in the Grothendieck group of \(\F(1,\dotsc,d)(E)\)
  \[
    [U_{d-1}^{\omega}/U_{d-1}]
    =
    [E]
    -
    \sum_{j=2}^{d}
    \big(
    [U_{d+1-j}/U_{d-j}]
    +
    [(U_{d+1-j}/U_{d-j})^{\vee}\otimes L]
    \big),
  \]
  which yields
  \[
    s_{1/t_{1}}(U_{d-1}^{\omega}/U_{d-1})
    =
    \prod_{j=2}^{d}
    (t_{1}-\xi_{j})(t_{1}+\xi_{j}+c_{1}(L))\,
    s_{1/t_{1}}(E)
    (1/t_{1})^{2(d-1)}.
  \]
  Using this expression of the Segre class in the above formula, and shifting by \(t_{1}^{2(d-1)}\), one gets
  \[
    \int_{\F^{\omega}(1,\dotsc,d)(E)}^{\F^{\omega}(1,\dotsc,d-1)(E)} 
    f(\xi_{1},\dotsc,\xi_{d})
    =
    \Big[t_{1}^{2n-1}\Big]
    \bigg(
    f(t_{1},\xi_{2},\dotsc,\xi_{d})
    \tprod_{1<j\leq d}
    (t_{1}-\xi_{j})(t_{1}+\xi_{j}+c_{1}(L))\,
    s_{1/t_{1}}(E)
    \bigg).
  \]
  One infers, using induction, that
  \[
    \int_{\F^{\omega}(1,\dotsc,d)(E)}^{X} 
    f(\xi_{1},\dotsc,\xi_{d})
    =
    [t_{1}^{2n-1}\dotsm t_{d}^{2n-1}]
    \bigg(
    f(t_{1},\dotsc,t_{d})
    \tprod_{1\leq i<j\leq d}
    (t_{i}-t_{j})(t_{i}+t_{j}+c_{1}(L))
    \tprod_{1\leq j\leq d} s_{1/t_{j}}(E)
    \bigg).
  \]

  \paragraph{From full flags to partial flags}
  Clearly, any flag inside an isotropic subbundle is an isotropic flag. 
  Let \(0=d_{0}<d_{1}<\dotsb<d_{m}=d\) be an increasing sequence of integers.
  Recall that on \(F\bydef\F^{\omega}(d_{1},\dots,d_{m})(E)\), there is the universal flag of vector bundles
  \[
    0\subsetneq U_{d_{1}}\subsetneq\dotsb\subsetneq U_{d_{m}}\subsetneq E
  \]
  where \(\rk(U_{d_{k}})=d_{k}\).
  The fiber product
  \[
    \mathbf{Y}
    \bydef
    \F(U_{d_{1}})\times_{F}\F(U_{d_{2}}/U_{d_{1}})\times_{F}\dotsb\times_{F}\F(U_{d_{m}}/U_{d_{m-1}})
  \]
  is isomorphic to \(\F^{\omega}(1,\dotsc,d)(E)\) with the natural projection map \(\F^{\omega}(1,\dotsc,d)(E)\to F\), and we get a commutative diagram
  \[
    \vcenter{
      \xymatrix@C=3cm{
        \F^{\omega}(1,\dotsc,d)(E)\ar[d]_{\pi''}\ar[r]_{\simeq}^{\theta}&
        \mathbf{Y}
        \ar[d]^{\pi'}
        \\
        X&
        F
        \ar[l]_{\pi}
      }
    }.
  \]

  The arguments developed in case \(A\) still apply, resulting in the same shift of the extracted monomial in the general formula.
  This concludes the proof.
\end{proof}

\section{Universal push-forward formulas for orthogonal flag bundles}
\label{se:BD}
We will need only few modifications to deal with the orthogonal setting. We will do almost the same reasoning, replacing the projective bundle of lines \(\P(E)\) by the quadric bundle of isotropic lines \(\Q(E)\).

\subsection{Definition of partial isotropic flag bundles of types \ensuremath{B} and \ensuremath{D}}  
\label{sse:defBD}
Let \(E\to X\) be a vector bundle of rank \(2n\) or \(2n+1\) equipped with a non-degenerate orthogonal form \(Q\colon E\otimes E\to L\) (with values in a certain line bundle \(L\to X\)).  We say that a subbundle \(S\) of \(E\) is isotropic if \(S\) is a subbundle of its orthogonal complement \(S^{\perp}\), where
\[
  S^{\perp}
  \bydef
  \{w\in E\mid \forall v\in S\colon Q(w,v)=0\}.
\]
In particular, if the dimension of \(E\) is even, \(S_{n}=S_{n}^{\perp}\).

Let \(1\leq d_{1}<\dotsb<d_{m}\leq n\) be a sequence of integers. We denote by \(\pi\colon \F^{Q}(d_{1},\dotsc,d_{m})(E)\to X\) the bundle of flags of isotropic subspaces of dimensions \(d_{1},\dotsc,d_{m}\) in the fibers of \(E\). On \(\F^{Q}(d_{1},\dotsc,d_{m})(E)\), there is a universal flag \(U_{d_{1}}\subsetneq\dotsb\subsetneq U_{d_{m}}\) of subbundles of \(\pi^{\ast}E\), where \(\rk(U_{d_{k}})=d_{k}\).

\subsection{Step-by-step construction of isotropic full flag bundles}
\label{sse:flagBD}
Let us recall the construction of \(\F^{Q}(1,2,\dotsc,d)(E)\) for \(d=1,2,\dotsc,n\) as as chains of quadric bundles of isotropic lines ~\cite[\S6]{EG}. 
We proceed over \(x\in X\) and write \(E=E_{x}\). Consider a flag of subspaces
\[
  0=V_{0}
  \subsetneq
  V_{1}
  \subsetneq
  \dotsb
  \subsetneq
  V_{n}
  \subsetneq
  E
\]
such that for each \(i\), the rank of \(V_{i}\) is \(i\) and \(V_{i}\) is isotropic.
In particular, if the dimension of \(E\) is even, \(V_{n}=V_{n}^{\perp}\).
For \(V_{1}\), we can take any isotropic line in \(E\). Thus, by definition 
\(\F^{Q}(1)(E)\simeq \Q(E)\), the codimension one subbundle of \(\P(E)\) cut out by \(Q\).
Next, we consider \(\F^{Q}(1,2)(E)\to\F^{Q}(1)(E)\) above \(V_{1}\). 
In order to get an isotropic subspace \(V_{1}\subsetneq V_{2}\subsetneq V_{1}^{\perp}\), it suffices to pick one more line in \(\Q(V_{1}^{\perp}/V_{1})\). 
Iterating this construction, 
\[
  \begin{array}{c}
    \F^{Q}(1,\dotsc,i,i+1)(E)\\
    \downarrow\\
    \F^{Q}(1,\dotsc,i)(E)
  \end{array}
  (\text{fiber}=\Q(V_{i}^{\perp}/V_{i}))
\]
picking one isotropic line in \(\Q(V_{i}^{\perp}/V_{i})\) at each step, one ends up 
with \(V_{n}^{\perp}/V_{n}\), which is either zero dimensional or one dimensional. 
In the first case, when the rank of \(E\) is even, it is usual to take only one of the two connected components, but we will not, as we do not want to treat this case separately.

Globalizing this construction over \(X\), we obtain
a chain of quadric bundles of isotropic lines
\[
  \F^{Q}(1,\dotsc,n)(E)
  \to
  \F^{Q}(1,\dotsc,n-1)(E)
  \to
  \dotsb
  \to
  \F^{Q}(1,2)(E)
  \to
  \F^{Q}(1)(E)
  \to
  X,
\]
which is the same as
\[
  \Q(U_{n-1}^{\perp}/U_{n-1})
  \to
  \Q(U_{n-2}^{\perp}/U_{n-2})
  \to
  \dotsb
  \to
  \Q(U_{1}^{\perp}/U_{1})
  \to
  \Q(E)
  \to
  X.
\]

\subsection{A push-forward formula for the quadric bundle}
\label{sse:QE}
The quadric bundle \(\iota\colon\Q(E)\hookrightarrow\P(E)\) of isotropic lines of \(E\),
cut out by the quadratic form \(Q\colon E\otimes E\to L\),
is the zero set of a section of the line bundle 
\[
  \Hom(\O_{\P(E)}(-1)\otimes\O_{\P(E)}(-1), L)
  \simeq
  \O_{\P(E)}(2)\otimes L.
\]
Its fundamental class in \(\P(E)\) is thus
\(
[\Q(E)]
=
2\xi+c_{1}(L)
\).
Let \(\pi\colon\P(E)\to X\) be the natural projection, and \(\rho=\pi\circ\iota\).
Let \(\tilde{\xi}\bydef \iota^{\ast}\xi\) denote the restriction of \(\xi=c_{1}(\O_{P(E)}(1))\).
One has
\[
  \rho_{\ast}\big(\tilde{\xi}^{i}\rho^{\ast}\alpha\big)
  =
  \pi_{\ast}\big([\Q(E)]\cdot\xi^{i}\pi^{\ast}\alpha\big)
  =
  \pi_{\ast}\big((2\xi+c_{1}(L))\xi^{i}\pi^{\ast}\alpha\big)
  =
  (2s_{i-r+2}(E)+c_{1}(L) s_{i-r+1}(E))\alpha,
\]
where \(r=\rk(E)\).
Similarly as we get~\eqref{eq:PE} from~\eqref{eq:push1}, we now infer
\begin{equation}
  \label{eq:QE}
  \rho_{\ast}f(\tilde{\xi})
  =
  [t^{r-1}]
  \big(
  f(t)
  (2t+c_{1}(L))
  s_{1/t}(E)
  \big).
\end{equation}

Since we are working with towers of quadric bundles in the orthogonal case, this formula will play the analog role as formula~\eqref{eq:PE} in this section.

\subsection{Universal push-forward formula for orthogonal flag bundles}
Let \((E,Q)\to X\) be a vector bundle of rank \(2n\) or \(2n+1\), equipped with a everywhere non-degenerate quadratic form \(Q\colon E\otimes E\to L\), for a certain line bundle \(L\to X\). 

Given a sequence of integers \(0=d_{0}<d_{1}<\dotsb<d_{m}\leq n\) as in Sect.~\ref{sse:defBD}, 
we set \(d\bydef d_{m}\) and for \(k=1,\dotsc,m\) we write \(r_{k}\bydef d_{k}-d_{k-1}\).
We also set for \(i=1,\dotsc,d\)
\[
  \xi_{i}
  \bydef
  -c_{1}\big(U_{d+1-i}/U_{d-i}\big).
\]

The following Gysin formula holds for the isotropic partial flag bundle \(\F^{Q}(d_{1},\dotsc,d_{m})(E)\to X\).
\begin{theo}
  \label{thm:BD}
  For any rational equivalence class
  \(
  f(\xi_{1},\dotsc,\xi_{d})
  \in
  A^{\bullet}(\F^{Q}(d_{1},\dotsc,d_{m})(E)),
  \)
  one has
  \begin{multline*}
    \int_{\F^{Q}(d_{1},\dotsc,d_{m})(E)}^{X}
    f(\xi_{1},\dotsc,\xi_{d})
    =\\
    \Big[
      {t_{1}}^{e_{1}}\dotsm{t_{d}}^{e_{d}}
    \Big]
    \bigg(
    f(t_{1},\dotsc,t_{d})\,
    \tprod_{1\leq i \leq d}
    (2t_{i}+c_{1}(L))
    \tprod_{1\leq i<j\leq d}
    (t_{i}-t_{j}) (t_{i}+t_{j}+c_{1}(L))
    \tprod_{1\leq i \leq d}
    s_{1/t_{i}}(E)
    \bigg),
  \end{multline*}
  where for \(j=d-d_{k}+i\) with \(i=1,\dotsc,r_{k}\),
  we denote
  \(e_{j}\bydef\rk(E)-i\).
\end{theo}
Note that, if the rank is \(2n\) and \(d=n\), we consider \emph{both} of the two isomorphic connected components of the flag bundle. Thus, if one is interested in only one of the two components, the result should be divided by \(2\). When \(c_{1}(L)=0\), this makes appear the usual coefficient \(2^{n-1}\).
\begin{proof}
  We first prove the formula for full flags. Since the quadratic form \(Q\) is everywhere non-degenerate, one can consider the isomorphism
  \(
  \iota_{Q}
  \colon
  v\in E\mapsto Q(v,\cdot)\in\Hom(E,L)
  \).
  As in the symplectic case, this isomorphism yields the relation in the Grothendieck group of \(\F^{Q}(1,\dotsc,j)(E)\)
  \[
    [U_{j}^{\perp}/U_{j}]
    =
    [E]
    -
    \sum_{i=1}^{j}
    \big(
    [U_{i}/U_{i-1}]
    +
    [(U_{i}/U_{i-1})^{\vee}\otimes L]
    \big).
  \]

  The proof goes in the very same way as for type \(C\), replacing formula~\eqref{eq:PE} for \(\P(E)\) by formula~\eqref{eq:QE} for \(\Q(E)\). So we skip the details.

  Then, to go from full flags to partial flags, the argument is the same as in the symplectic case.
  Clearly, any flag inside an isotropic subbundle is an isotropic flag. 
  It yields the expected formulas with shifted exponents.
\end{proof}

Note that in the basic case where the quadratic form \(Q\) takes values in the trivial line bundle \(L=\O_{X}\) the theorem has a simpler form, and reads as follows.
\begin{coro}
  When \(L=\O_{X}\),
  for any rational equivalence class
  \(
  f(\xi_{1},\dotsc,\xi_{d})
  \in
  A^{\bullet}(\F^{Q}(d_{1},\dotsc,d_{m})(E)),
  \)
  one has
  \[
    \int_{\F^{Q}(d_{1},\dotsc,d_{m})(E)}^{X}
    f(\xi_{1},\dotsc,\xi_{d})
    =
    2^{d}
    \Big[
      {t_{1}}^{e_{1}}\dotsm{t_{d}}^{e_{d}}
    \Big]
    \bigg(
    f(t_{1},\dotsc,t_{d})\,
    \tprod_{1\leq i<j\leq d}
    (t_{i}^{2}-t_{j}^{2})
    \tprod_{1\leq i \leq d}
    s_{1/t_{i}}(E)
    \bigg),
  \]
  where for \(j=d-d_{k}+i\) with \(i=1,\dotsc,r_{k}\),
  we denote
  \(e_{j}\bydef\rk(E)-1-i\).
\end{coro}

\section{Applications: new determinantal formulas}
\label{se:formulas}
To finish, we would like to compute new determinantal formulas, for types \(A,B,C,D\), in order to illustrate the usefulness and the efficiency of our approach. 

We shall use the following linearity result, whose proof is left to the reader.
\begin{lemm}
  \label{lem:linearity}
  For any \(f_{ij}\in A^{\bullet}X[t_{j}]\) where \(1\leq i,j\leq d\),
  and for any exponents \(e_{1},\dotsc,e_{d}\in\mathbb{N}\)
  \[
    [{t_{1}}^{e_{1}}\dotsm{t_{d}}^{e_{d}}]
    \big(\det(f_{ij})\big)
    =
    \det\big([t_{j}^{e_{j}}](f_{ij})\big).
  \]
\end{lemm}
\subsection{Schur functions}
For any partition \(\lambda=(\lambda_{1},\dotsc,\lambda_{d})\),
recall (\cite{Macdonald,Fulton}) that the Schur polynomial \(s_{\lambda}\in \Z[t_{1},\dotsc,t_{d}]\) 
can be defined by the formula:
\begin{equation}
  \label{eq:Schur}
  s_{\lambda}(t_{1},\dotsc,t_{d})
  \bydef
  \det({t_{j}}^{\lambda_{i}+d-i})_{1\leq i,j \leq d}
  \,\Big/\,
  {\tprod}_{1\leq i<j \leq d}(t_{i}-t_{j}).
\end{equation}
Note that in particular for \(\lambda=(i)\), one has \(s_{i}=h_{i}\) the
complete symmetric function of degree \(i\), and accordingly 
\(s_{i}(E)=s_{i}(\xi_{1},\dotsc,\xi_{d})\).
Then the Jacobi--Trudi identity states
\[
  s_{\lambda}(t_{1},\dotsc,t_{d})
  =
  \det\big(s_{\lambda_{i}-i+j}(t_{1},\dotsc,t_{d})\big)_{1\leq i,j\leq d},
\]
where \(s_{i}=0\) for \(i<0\).
This identity allows one to generalize the Segre classes of \(E\) for any sequence of integers \(\lambda\in\Z^{d}\) in a natural way, by setting
\begin{equation}
  \label{eq:s(E)}
  s_{\lambda}(E)
  \bydef
  \det\big(s_{\lambda_{i}+(j-i)}(E)\big)_{1\leq i,j\leq d}.
\end{equation}
Note that there is a dichotomy for sequences \(\lambda\in\Z^{d}\):
\begin{itemize}
  \item 
    either for all \(i=1,\dotsc,d\), one has \(\lambda_{i}\geq i-d\),
    then~\eqref{eq:Schur} holds for \(s_{\lambda}\);
  \item 
    or for some \(i\in 1,\dotsc,d\), one has \(\lambda_{i}< i-d\),
    then for \(j=1,\dotsc,d\) one has \(s_{\lambda_{i}+(j-i)}(E)=0\) and accordingly \(s_{\lambda}(E)=0\).
\end{itemize}

In what follows, the difference of two sequences in \(\Z^{d}\) is defined componentwise.

\subsection{Formulas using monomials}
These Schur functions naturally appear for type \(A\) when one uses the additive basis of monomials, as illustrated below.

Let \(E\to X\) be a rank \(n\) vector bundle.
\begin{prop}
  \label{prop:monomialsA}
  A rational equivalence class on the partial flag bundle 
  \(\F(d_{1},\dotsc,d_{m})(E)\to X\)
  \[
    f(\xi_{1},\dotsc,\xi_{d})
    \bydef
    \sum_{\lambda\in\Z^{d}}
    \alpha_{\lambda}
    {\xi_{1}}^{\lambda_{1}}\dotsm{\xi_{d}}^{\lambda_{d}}
  \]
  goes via the Gysin map to
  \[
    \int_{\F(d_{1},\dotsc,d_{m})(E)}^{X}
    f(\xi_{1},\dotsc,\xi_{d})
    =
    \sum_{\lambda\in\Z^{d}}
    \alpha_{\lambda}
    s_{\lambda-\nu}(E),
  \]
  where \(\nu\) is an increasing sequence of integers determined by the flag data, \textit{viz.}
  \[
    \nu_{i}
    \bydef
    n-d_{k}
    \qquad \text{for }
    d-d_{k}<i\leq d-d_{k-1}.
  \]
\end{prop}
\begin{proof}
  The push-forward formula for type \(A\) yields
  \[
    \int_{\F(d_{1},\dotsc,d_{m})(E)}^{X}
    f(\xi_{1},\dotsc,\xi_{d})
    =
    \Big[{t_{1}}^{e_{1}}\dotsm{t_{d}}^{e_{d}}\Big]
    \bigg(
    \tsum
    \alpha_{\lambda}
    t_{1}^{\lambda_{1}}\dotsm t_{d}^{\lambda_{d}}
    \tprod_{1\leq i<j\leq d}
    (t_{i}-t_{j})
    \tprod_{1\leq j\leq d} 
    s_{1/t_{j}}(E)
    \bigg),
  \]
  where
  \(e_{j}=n+d-d_{k}-j\), 
  for
  \(d-d_{k}<j\leq d-d_{k-1}\).

  Now, it is well known (and due to Vandermonde) that 
  \(
  \prod_{1\leq i<j \leq d}
  (t_{i}-t_{j})
  =
  \det(t_{j}^{d-i})_{1\leq i,j\leq d}
  \),
  and then by Lemma~\ref{lem:linearity}
  \[
    \int_{\F(d_{1},\dotsc,d_{m})(E)}^{X}
    f(\xi_{1},\dotsc,\xi_{d})
    =
    \sum
    \alpha_{\lambda}
    \det\Big(
    [t_{j}^{e_{j}}](t_{j}^{d-i+\lambda_{j}}s_{1/t_{j}}(E))
    \Big)_{1\leq i,j\leq d}
    =
    \sum
    \alpha_{\lambda}
    \det\left(
    s_{\lambda_{j}+d-i-e_{j}}(E)
    \right)_{1\leq i,j\leq d}.
  \]
  Since for \(d-d_{k}<j\leq d-d_{k-1}\), one has
  \[
    \lambda_{j}+d-i-e_{j}
    =
    \lambda_{j}-(n-d_{k})+j-i
    =
    \lambda_{j}-\nu_{j}+j-i.
  \]
  A plain transposition of the determinant yields the announced expression.
\end{proof}

In the symplectic case and in the orthogonal case, we will now assume that \(L=\O_{X}\) is the trivial bundle.
In this case
\[
  \prod_{1\leq i<j \leq d}(t_{i}-t_{j})(t_{i}+t_{j}+c_{1}(L))
  =
  \prod_{1\leq i<j \leq d}(t_{i}^{2}-t_{j}^{2})
  =
  \det\big(t_{j}^{2(d-i)}\big)_{1\leq i,j\leq d}.
\]
Hence, the proof of Proposition \ref{prop:monomialsA} is easily adapted in order to obtain determinantal formulas for types \(B\), \(C\), \(D\).
We first introduce the classes
\begin{equation}
  \label{eq:s(2)(E)}
  s_{\lambda}^{(2)}(E)
  \bydef
  \det\big(s_{\lambda_{i}+2(j-i)}(E)\big)_{1\leq i,j\leq d},
\end{equation}
for all sequences of integers \(\lambda\in\Z^{d}\).
It seems adequate to call them \textsl{quadratic Schur functions} (compare with~\eqref{eq:s(E)}).
They are closely related to the classes \(s_{\lambda}^{[2]}(E)\) defined in \cite{PR},
as it will soon appear (see Section \ref{sse:push_Schur}).

Let \(E\to X\) be a symplectic vector bundle of rank \(2n\), equipped with the symplectic form \(\omega\colon E\otimes E\to \O_{X}\). 
\begin{prop}
  A rational equivalence class
  on the isotropic flag bundle \(\F^{\omega}(d_{1},\dotsc,d_{m})(E)\to X\)
  \[
    f(\xi_{1},\dotsc,\xi_{d})
    \bydef
    \sum_{\lambda\in\Z^{d}}\alpha_{\lambda}\xi_{1}^{\lambda_{1}}\dotsm\xi_{d}^{\lambda_{d}}
  \]
  goes via the Gysin map to
  \[
    \int_{\F^{\omega}(d_{1},\dotsc,d_{m})(E)}^{X}
    f(\xi_{1},\dotsc,\xi_{d})
    =
    \sum_{\lambda\in\Z^{d}}
    \alpha_{\lambda}
    s_{\lambda-\nu}^{(2)}(E),
  \]
  where \(\nu\) is an increasing sequence of integers determined by the flag data, \textit{viz.}
  \[
    \nu_{i}
    =
    2n-d-d_{k}+i
    \qquad\text{for }
    d-d_{k}<i\leq d-d_{k-1}.
  \]
\end{prop}

Let \(E\to X\) be an orthogonal vector bundle, equipped with the quadratic form \(Q\colon E\otimes E\to \O_{X}\). 
\begin{prop}
  A rational equivalence class
  on the isotropic flag bundle \(\F^{Q}(d_{1},\dotsc,d_{m})(E)\to X\)
  \[
    f(\xi_{1},\dotsc,\xi_{d})
    \bydef
    \sum_{\lambda\in\Z^{d}}\alpha_{\lambda}\xi_{1}^{\lambda_{1}}\dotsm\xi_{d}^{\lambda_{d}}
  \]
  goes via the Gysin map to
  \[
    \int_{\F^{Q}(d_{1},\dotsc,d_{m})(E)}^{X}
    f(\xi_{1},\dotsc,\xi_{d})
    =
    2^{d}
    \sum_{\lambda\in\Z^{d}}
    \alpha_{\lambda}
    s_{\lambda-\nu}^{(2)}(E),
  \]
  where \(\nu\) is an increasing sequence of integers determined by the flag data, \textit{viz.}
  \[
    \nu_{i}
    =
    (\rk(E)-1)-d-d_{k}+i
    \qquad\text{for }
    d-d_{k}<i\leq d-d_{k-1}.
  \]
\end{prop}

\subsection{Formulas using Schur functions}
\label{sse:push_Schur}
We now focus on the Grassmann bundles and use the additive basis of Schur functions.
We start with type \(A\), and show an alternative deduction of \cite[Corollary 1]{JLP}.

\begin{prop}
Let \(E\to X\) be a rank \(n\) vector bundle.
  For \(d=1,\dotsc,n-1\), for any \(\lambda\in\Z^{d}\), one has the Gysin formula
  \[
    \int_{\F(d)(E)}^{X}
    s_{\lambda}(\xi_{1},\dotsc,\xi_{d})
    =
    s_{\lambda-(n-d)^{d}}(E).
  \]
\end{prop}
\begin{proof}
  The push-forward formula for the polynomial \(s_{\lambda}(\xi_{1},\dotsc,\xi_{d})\) is
  \[
    \int_{\F(d)(E)}^{X}
    s_{\lambda}(\xi_{1},\dotsc,\xi_{d})
    =
    \Big[{t_{1}}^{e_{1}}\dotsm{t_{d}}^{e_{d}}\Big]
    \bigg(
    s_{\lambda}(t_{1},\dotsc,t_{d})
    \tprod_{1\leq i<j\leq d}
    (t_{i}-t_{j})
    \tprod_{1\leq j\leq d} s_{1/t_{j}}(E)
    \bigg),
  \]
  where for \(j=1,\dotsc,d\), the exponents are
  \(e_{j}=n-j\).

  If for some \(i\in1,\dotsc,d\), one has \(\lambda_{i}<i-d\), then one checks that the stated formula becomes \(\int_{\F(d)(E)}^{X}0=0\). So we can assume that definition~\eqref{eq:Schur} holds:
  \[
    s_{\lambda}(t_{1},\dotsc,t_{d})
    \tprod_{1\leq i<j \leq d}
    (t_{i}-t_{j})
    =
    \det({t_{j}}^{\lambda_{i}+d-i})_{1\leq i,j\leq d},
  \]
  and then 
  by linearity of the determinant with respect to the columns and
  by Lemma~\ref{lem:linearity}
  \[
    \int_{\F(d)(E)}^{X}
    s_{\lambda}(\xi_{1},\dotsc,\xi_{d})
    =
    \det\big(
    [{t_{j}}^{e_{j}}]({t_{j}}^{\lambda_{i}+d-i}s_{1/t_{j}}(E))
    \big)_{1\leq i,j\leq d}
    =
    \det\big(
    s_{\lambda_{i}+d-i-e_{j}}(E)
    \big)_{1\leq i,j\leq d},
  \]
  which is the announced Schur function since
  \[
    \lambda_{i}+d-i-e_{j}
    =
    \lambda_{i}-(n-d)+(j-i).
    \qedhere
  \]
\end{proof}

We now treat types \(B\), \(C\), \(D\). 
Here the quadratic Schur functions \(s_{\lambda}^{(2)}\) will appear again.
The comparison with the results of \cite{PR} (for the case of maximal rank \(d=n\); see also \cite{Zie2}), reveals a deep connection between the Schur type functions \(s_{\lambda}^{(2)}(E)\) of this work and the Schur type functions \(s_{\lambda}^{[2]}(E)\) of \cite{PR}, that should probably be investigated.

But first, we need to prove the following combinatorial result.
\begin{lemm}
  \label{lem:e-j}
  For \(e\in\Z\) and  \(\Lambda(t_{1},\dotsc,t_{d})\) antisymmetric, one has
  \[
    \Big[\tprod_{j=1}^{d}t_{j}^{e-j}\Big]
    \bigg(
    \Lambda(t_{1},\dotsc,t_{d})
    \tprod_{1\leq i<j\leq d}(t_{i}+t_{j})
    \bigg)
    =
    \Big[\tprod_{j=1}^{d}t_{j}^{e-j}\Big]
    \bigg(
    \Lambda(t_{1},\dotsc,t_{d})
    \tprod_{j=1}^{d}t_{j}^{j-1}
    \bigg).
  \]
\end{lemm}
\begin{proof}
  We prove this formula by induction on \(d\).
  Denote by \((i,j)\prec (i_{0},j_{0})\) the set of pairs of integers \(i<j\) such that \((i,j)\) is smaller than \((i_{0},j_{0})\) in inverse lexicographic order.
  Notice that
  \[
    \prod_{1\leq i<j\leq d}(t_{i}+t_{j})
    =
    \sum_{p=1}^{d-1}
    t_{d}^{p-1}
    t_{d-p}
    \tprod_{(i,j)\prec(d-p,d)}(t_{i}+t_{j})
    +
    t_{d}^{d-1}
    \prod_{1\leq i<j\leq d-1}(t_{i}+t_{j}).
  \]
  Hence:
  \begin{multline*}
    \Big[\tprod_{j=1}^{d}t_{j}^{e-j}\Big]
    \bigg(
    \Lambda(t_{1},\dotsc,t_{d})
    \tprod_{1\leq i<j\leq d}(t_{i}+t_{j})
    \bigg)
    =\\
    \sum_{p=1}^{d-1}
    \Big[\tprod_{j=1}^{d}t_{j}^{e-j}\Big]
    \bigg(
    t_{d}^{p-1}t_{d-p}
    \Lambda(t_{1},\dotsc,t_{d})
    \tprod_{(i,j)\prec(d-p,d)}(t_{i}+t_{j})
    \bigg)
    +
    \Big[\tprod_{j=1}^{d}t_{j}^{e-j}\Big]
    \bigg(
    t_{d}^{d-1}
    \Lambda(t_{1},\dotsc,t_{d})
    \tprod_{1\leq i<j\leq d-1}(t_{i}+t_{j})
    \bigg).
  \end{multline*}
  All the \(d-1\) first terms in the sum are \(=0\), since after shifting by \(t_{d-p}\), one takes the coefficient of a monomial symmetric in \(t_{d-p}\) and \(t_{d-p+1}\) in an antisymmetric function in \(t_{d-p}\) and \(t_{d-p+1}\). Thus
  \[
    \Big[\tprod_{i=1}^{d}t_{i}^{e+i}\Big]
    \bigg(
    \Lambda(t_{1},\dotsc,t_{d})
    \tprod_{1\leq i<j \leq d}(t_{i}+t_{j})
    \bigg)
    =
    \Big[\tprod_{j=1}^{d}t_{j}^{e-j}\Big]
    \bigg(
    t_{d}^{d-1}
    \Lambda(t_{1},\dotsc,t_{d})
    \tprod_{1\leq i<j\leq d-1}(t_{i}+t_{j})
    \bigg).
  \]
  One concludes by induction on \(d\), by considering \([t_{d}^{e-d}]\big(
  t_{d}^{d-1}
  \Lambda(t_{1},\dotsc,t_{d})\big)\) as an antisymmetric function in \(t_{1},\dotsc,t_{d-1}\).
\end{proof}

We can now proceed.
\begin{prop}
  \label{prop:gysin-schur}
  Let \(E\to X\) be a symplectic vector bundle of rank \(2n\), equipped with the symplectic form \(\omega\colon E\otimes E\to \O_{X}\). 
  For \(d=1,\dotsc,n\), for any \(\lambda\in\Z^{d}\), one has the Gysin formula
  \[
    \int_{\F^{\omega}(d)(E)}^{X}
    s_{\lambda}(\xi_{1},\dotsc,\xi_{d})
    =
    s_{\lambda-\mu}^{(2)}(E),
  \]
  where \(\mu_{i}=2n-d+1-i\) for \(i=1,\dotsc,d\).
\end{prop}
\begin{proof}
  The push-forward formula for the polynomial \(s_{\lambda}(\xi_{1},\dotsc,\xi_{d})\) is
  \[
    \int_{\F^{\omega}(d)(E)}^{X}
    s_{\lambda}(\xi_{1},\dotsc,\xi_{d})
    =
    \Big[{t_{1}}^{e_{1}}\dotsm{t_{d}}^{e_{d}}\Big]
    \bigg(
    s_{\lambda}(t_{1},\dotsc,t_{d})
    \tprod_{1\leq i<j\leq d}
    (t_{i}-t_{j})
    \tprod_{1\leq i<j\leq d}
    (t_{i}+t_{j})
    \tprod_{1\leq j\leq d} s_{1/t_{j}}(E)
    \bigg),
  \]
  where for \(j=1,\dotsc,d\), the exponents are
  \(e_{j}\bydef 2n-j\).

  If for some \(i\in1,\dotsc,d\), one has \(\lambda_{i}<i-d\), then one checks that the stated formula becomes \(\int_{\F^{\omega}(d)(E)}^{X}0=0\). So we can assume that definition~\eqref{eq:Schur} holds:
  \[
    s_{\lambda}(t_{1},\dotsc,t_{d})
    \tprod_{1\leq i<j \leq d}
    (t_{i}-t_{j})
    =
    \det({t_{j}}^{\lambda_{i}+d-i})_{1\leq i,j\leq d},
  \]
  and then by linearity of the determinant with respect to the columns
  \[
    \int_{\F^{\omega}(d)(E)}^{X}
    s_{\lambda}(\xi_{1},\dotsc,\xi_{d})
    =
    \Big[{t_{1}}^{e_{1}}\dotsm{t_{d}}^{e_{d}}\Big]
    \bigg(
    \det(
    {t_{j}}^{\lambda_{i}+d-i}s_{1/t_{j}}(E)
    )_{1\leq i,j\leq d}\;
    \tprod_{1\leq i<j\leq d}
    (t_{i}+t_{j})
    \bigg),
  \]
  which according to Lemma~\ref{lem:e-j} is the same as
  \[
    \int_{\F^{\omega}(d)(E)}^{X}
    s_{\lambda}(\xi_{1},\dotsc,\xi_{d})
    =
    \Big[{t_{1}}^{e_{1}}\dotsm{t_{d}}^{e_{d}}\Big]
    \bigg(
    \det(
    {t_{j}}^{\lambda_{i}+d-i}s_{1/t_{j}}(E)
    )_{1\leq i,j\leq d}\;
    \tprod_{1\leq j\leq d}t_{j}^{j-1}
    \bigg).
  \]
  Thus, using again the linearity with respect to the columns and 
  Lemma~\ref{lem:linearity}
  \[
    \int_{\F^{\omega}(d)(E)}^{X}
    s_{\lambda}(\xi_{1},\dotsc,\xi_{d})
    =
    \det\Big(
    [{t_{j}}^{e_{j}}]
    (
    {t_{j}}^{\lambda_{i}+d-i+j-1}s_{1/t_{j}}(E)
    )
    \Big)_{1\leq i,j\leq d}
    =
    \det\big(
    s_{\lambda_{i}+d-i+j-1-e_{j}}(E)
    \big)_{1\leq i,j\leq d},
  \]
  which is the announced quadratic Schur function since
  \[
    \lambda_{i}+d-i+j-1-e_{j}
    =
    \lambda_{i}-(2n-d+1-i)+2(j-i)
    =
    \lambda_{i}-\mu_{i}+2(j-i).
    \qedhere
  \]
\end{proof}
The partition \(\mu\) is the maximal strict partition in \((2n-d)^{d}\) and in the case where \(d=n\), it is \(\rho=(d,d-1,\dotsc,1)\).
As already noticed in the case \(d=n\) in \cite[Theorem~5.13]{PR}, if one of the \(\lambda_{i}-\mu_{i}\) is odd, one obtains \(0\). 
Indeed in one row of the determinantal form \eqref{eq:s(2)(E)} of quadratic Schur functions the degree jumps by \(2\) between two columns, so all the degrees in the \(i\)th row would be odd;
but all odd Segre classes \(s_{2p+1}(E)\) are \(=0\) when \(L=\O_{X}\).

The proof of Proposition~\ref{prop:gysin-schur} is easily adapted in the orthogonal case, in order to get the following statement.

\begin{prop}
  Let \(E\to X\) be an orthogonal vector bundle, equipped with the quadratic form \(Q\colon E\otimes E\to \O_{X}\). 
  For \(d=1,\dotsc,\lfloor\rk(E)/2\rfloor\),
  for \(\lambda\in\Z^{d}\), one has the following Gysin formula
  \[
    \int_{\F^{Q}(d)(E)}^{X}
    s_{\lambda}(\xi_{1},\dotsc,\xi_{d})
    =
    2^{d}\,s_{\lambda-\mu}^{(2)}(E),
  \]
  where \(\mu_{i}=\rk(E)-d-i\) for \(i=1,\dotsc,d\).
\end{prop}
\backmatter
\paragraph*{Acknowledgment.}
We thank Tomoo Matsumura for pointing out a mistake in a former version of Sect. \ref{sse:QE}.
\bibliographystyle{$LATEX/smfplain}
\bibliography{$LATEX/bib/gysin}
\end{document}